\documentclass[12pt,onesided]{article}

\usepackage[leqno]{amsmath}
\usepackage{amsthm}
\usepackage{amssymb}
\usepackage{amsfonts}
\usepackage{array}
\usepackage{hyperref}

\newtheorem{theorem}{THEOREM}
\newtheorem{corollary}{COROLLARY}

\def\bi{\begin{itemize}}
\def\ei{\end{itemize}}

\begin{document}

\title{Existence, uniqueness, and minimality of\\ the Jordan measure decomposition}
\author{Tom Fischer\thanks{Institute of Mathematics, University of Wuerzburg, 
Campus Hubland Nord, Emil-Fischer-Strasse 30, 97074 Wuerzburg, Germany.
Tel.: +49 931 3188911.
E-mail: {\tt tom.fischer@uni-wuerzburg.de}.
}\\
University of Wuerzburg}
\date{\today}

\maketitle

\begin{abstract}
This is a note of purely didactical purpose as the proof of the Jordan measure decomposition
is often omitted in the related literature. Elementary proofs are provided for the existence, the 
uniqueness, and the minimality property of the Jordan decomposition for a finite signed measure.
\end{abstract}


For the following recall that the symmentric difference $A\triangle B$ of two sets $A$ and $B$ is
defined by $A\triangle B = (A \cup B)\setminus(A\cap B)$.

\begin{theorem}[Hahn Decomposition]
\label{H}
For a measurable space $(\Omega, \mathcal{A})$ and a finite signed measure $\mu$
on $(\Omega, \mathcal{A})$ there exists $P\in\mathcal{A}$ and 
$N = \Omega\setminus P\in\mathcal{A}$
such that for any $A\in\mathcal{A}$ with $A\subset P$ one has that $\mu(A) \geq 0$, 
and for any $A\in\mathcal{A}$ with $A\subset N$ one has that $\mu(A) \leq 0$.  
\end{theorem}


The pair $(P,N)$ is called a {\em Hahn decomposition} of $\mu$.
The decomposition is {\em not unique}, but essentially unique in the sense that for a second
decomposition $(P',N')$ one has that the symmetric differences
$P\triangle P'$ and $N\triangle N'$ are null sets of $\mu$ in
the strong sense that for any $A\in\mathcal{A}$ with $A\subset P\triangle P'$ ($A\subset N\triangle N'$)
one has $\mu(A)=0$. See textbooks such as Billingsley (1995) or Bauer (1992) for a proof of Theorem \ref{H}.

\begin{theorem}[Jordan Decomposition]
\label{J}
Every finite signed measure $\mu$ has a unique decomposition into a difference 
$\mu = \mu^+ - \mu^-$ of two non-negative finite measures $\mu^+$ and $\mu^-$
such that for any Hahn decomposition $(P,N)$ of $\mu$ one has
for $A\in\mathcal{A}$ that 
$\mu^+(A) = 0$ if $A \subset N$ and $\mu^-(A) = 0$ if $A \subset P$. 
\end{theorem}

The pair $(\mu^+, \mu^-)$ is called the {\em Jordan decomposition} of $\mu$.
Note that the Jordan decomposition is unique, while the Hahn decomposition is only essentially unique.

\begin{proof}[Proof of Theorem \ref{J}]
Existence:
Let $(P,N)$ be a Hahn decomposition of $\mu$ by Theorem \ref{H} and
for all $A\in\mathcal{A}$ define $\mu^+$ and $\mu^-$ by
\begin{equation}
\label{01}
\mu^+(A) = \mu(A \cap P)
\end{equation}
and
\begin{equation}
\label{02}
\mu^-(A) = - \mu(A \cap N) .
\end{equation}
Clearly, $\mu = \mu^+ - \mu^-$, and $\mu^+$ and $\mu^-$ are non-negative finite measures on $(\Omega, \mathcal{A})$.
We now have to show that $\mu^+(A) = 0$ if $A \subset N'$, and $\mu^-(A) = 0$ if $A \subset P'$ for {\em any}
Hahn decomposition $(P',N')$ of $\mu$. Assume first $A \subset N'$. 
Observe that
\begin{equation}
P = (P \cap P') \cup (P \cap (P \triangle P')) 
\end{equation}
and
\begin{equation}
(P \cap P') \cap (P \cap (P \triangle P')) = \emptyset .
\end{equation}
Now, 
\begin{equation}
\mu^+(A) 
= 
\mu(A \cap P)
=
\mu(A \cap (P \cap P')) + \mu(A \cap (P \cap (P \triangle P')))
= 0 ,
\end{equation}
as $A \cap (P \cap P') = \emptyset$,
and 
$A \cap (P \cap (P \triangle P'))$ contained in $P \triangle P'$ and therefore a $\mu$ null set. 
The statement for $A \subset P'$ follows in a similar manner.\\
Uniqueness:
Consider now a Jordan decomposition $\mu = \mu^+ - \mu^-$
and an arbitrary Hahn decomposition $(P',N')$ of $\mu$. For any element $A\in\mathcal{A}$ 
for which $A \subset P$,
one obtains $\mu(A) = \mu^+(A) - \mu^-(A) = \mu^+(A)$. 
For an arbitrary $A\in\mathcal{A}$ one therefore obtains
\begin{equation}
\label{1}
\mu^+(A)
=
\mu^+(A \cap P') + \mu^+(A \cap N') 
= 
\mu^+(A \cap P')
=
\mu(A \cap P') .
\end{equation}
Likewise,
\begin{equation} 
\label{4}
\mu^-(A)
=
- \mu(A \cap N')
\end{equation}
can be shown. Hence, any two Jordan decompositions are identical
and $(\mu^+, \mu^-)$ is uniquely determined by the choice of $\mu$.
\end{proof}

Note that as \eqref{01} and \eqref{02} defines a Jordan decomposition,
\eqref{1} and \eqref{4} also show that 
\begin{eqnarray}
\mu(A \cap P) 
& = &
\mu(A \cap P') \quad \text{ and}\\
\mu(A \cap N)
& = &
\mu(A \cap N')
\end{eqnarray}
for any two (distinct) Hahn decompositions $(P,N)$ and $(P',N')$ of $\mu$.

\begin{corollary}
Given a Hahn decomposition $(P,N)$,
if two non-negative finite measures $\mu^+$ and $\mu^-$ with 
$\mu = \mu^+ - \mu^-$
have the property that, for $A\in\mathcal{A}$,
$\mu^+(A) = 0$ if $A \subset N$ and $\mu^-(A) = 0$ if $A \subset P$, then 
$(\mu^+, \mu^-)$ is the Jordan decomposition of $\mu$.
\end{corollary}

Therefore, when in search for a Jordan decomposition, the described property only needs 
to be shown for one Hahn decomposition.

\begin{proof}
Similar to Eq.~\eqref{1}, it holds that $\mu^+(A) = \mu(A \cap P)$
and $\mu^-(A) = - \mu(A \cap N)$, $A\in\mathcal{A}$,
under the conditions and for the Hahn decomposition $(P,N)$ of the corollary.
\end{proof}

\begin{corollary}
Given a Jordan decomposition $(\mu^+, \mu^-)$ of a finite signed measure $\mu$,
\begin{eqnarray}
\label{10}
\mu^+(A) & = & \sup_{B\in\mathcal{A}, B\subset A} \mu(B) \\
\label{11}
\mu^-(A) & = & -\inf_{B\in\mathcal{A} B\subset A} \mu(B) 
\end{eqnarray}
for any $A\in\mathcal{A}$. Furthermore, if $\mu = \nu^+ - \nu^-$ for
a pair of finite non-negative measures $(\nu^+, \nu^-)$, then
\begin{equation}
\label{13}
\nu^+ \geq \mu^+ \text{ and } \nu^- \geq \mu^- .
\end{equation}
\end{corollary}

Eq.~\eqref{13} means that the Jordan decomposition $(\mu^+, \mu^-)$ is the 
{\em minimal} decomposition of $\mu$ into a difference of two non-negative measures.
This is also called the {\em minimality property} of the Jordan decomposition.

\begin{proof}
For $B\in\mathcal{A}, B\subset A$, and any Hahn decomposition $(P,N)$, one has
\begin{equation}
\mu(B) = \mu^+(B) - \mu^-(B) \leq \mu^+(B) \leq \mu^+(A) = \mu(A\cap P) ,
\end{equation}
which, by $A\cap P\subset A$, implies \eqref{10}. Similarly, 
\begin{equation}
\mu(B) = \mu^+(B) - \mu^-(B) \geq - \mu^-(B) \geq - \mu^-(A) = -\mu(A\cap N) ,
\end{equation}
implies \eqref{11}. The second part we show by contradiction, hence assume
for some $A\in\mathcal{A}$ that $\nu^+(A) < \mu^+(A)$. 
Therefore, $\nu^+(A\cap P) + \nu^+(A\cap N) < \mu^+(A\cap P) + \mu^+(A\cap N)$.
It follows that
\begin{equation}
\label{12}
\nu^+(A\cap P) < \mu^+(A\cap P) ,
\end{equation}
since $\mu^+(A\cap N) = 0$ and 
$\nu^+(A\cap N) \geq 0$. Because of \eqref{12}, we must have 
$\nu^-(A\cap P) < \mu^-(A\cap P) = 0$, which is a contradiction to the non-negativity of
$\nu^-$. The proof of $\nu^- \geq \mu^-$ works similar.
\end{proof}



\end{document}